\documentclass[3p,10pt,a4paper,twoside,fleqn,sort&compress]{filomat}
\usepackage{amssymb,amsmath,latexsym}
\usepackage[varg]{pxfonts}
\usepackage{graphicx}
\usepackage{pst-all}

\newtheorem{theorem}{Theorem}[section]

\newtheorem{corollary}[theorem]{Corollary}

\newtheorem{definition}[theorem]{Definition}
\newtheorem{example}[theorem]{Example}

\newtheorem{lemma}[theorem]{Lemma}

\newtheorem{remark}[theorem]{Remark}

\newcommand{\e}{\varepsilon}
\newcommand{\N}{\mathbb{N}}

\newcommand{\sub}{\subseteq}

\newcommand{\ov}{\overline}

\newcommand{\rg}{\rightarrow}
\newcommand{\lo}{\longrightarrow}

\newcommand{\wt}{\widetilde}
\newcommand{\vf}{\varphi}
\newcommand{\fr}{\frac}
\newcommand{\al}{\alpha}
\newcommand{\la}{\lambda}
\newcommand{\pa}{\partial}
\newcommand{\bt}{\beta}

\newcommand{\FilVol}{xx}
\newcommand{\FilYear}{yyyy}
\newcommand{\FilPages}{zzz--zzz}
\newcommand{\FilDOI}{(will be added later)}
\newcommand{\FilReceived}{Received: dd Month yyyy; Revised: dd Month yyyy; Accepted: dd Month yyyy}

\begin{document}

\title{On locally 1-connectedness of quotient spaces and its applications to fundamental groups}

\author[affil1]{Ali Pakdaman}
\ead{a.pakdaman@gu.ac.ir}
\author[affil2]{Hamid Torabi}
\ead{hamid$_{-}$torabi86@yahoo.com}
\author[affil2]{Behrooz Mashayekhy}
\ead{bmashf@um.ac.ir}
\address[affil1]{Department of Mathematics, Faculty of Sciences, Golestan University, P.O.Box 155, Gorgan, Iran.}
\address[affil2]{Department of Pure Mathematics, Center of Excellence in Analysis on Algebraic Structures, Ferdowsi University of Mashhad, P.O.Box 1159-91775, Mashhad, Iran.}

\newcommand{\AuthorNames}{A. Pakdaman et al.}

\newcommand{\FilMSC}{Primary 55P65; Secondary 55Q52, 55Q70}
\newcommand{\FilKeywords}{Locally 1-connected, Quotient space, Quasitopological fundamental group}
\newcommand{\FilCommunicated}{Ljubisa Kocinac}
\newcommand{\FilSupport}{ }

\begin{abstract}
Let $X$ be a locally 1-connected metric space and $A_1,A_2,...,A_n$ be connected, locally path connected and compact pairwise disjoint subspaces of $X$. In this paper, we show that the quotient space $X/(A_1,A_2,...,A_n)$ obtained from $X$ by collapsing each of the sets $A_i$'s to a point, is also locally 1-connected. Moreover, we prove that the induced continuous homomorphism of quasitopological fundamental groups is surjective. Finally, we give some applications to find out some properties of the fundamental group of the quotient space $X/(A_1,A_2,...,A_n)$.
\end{abstract}

\maketitle

\makeatletter
\renewcommand\@makefnmark%
{\mbox{\textsuperscript{\normalfont\@thefnmark)}}}
\makeatother

\section{Introduction and Motivation}
Let $\sim$ be an equivalent relation on a topological space $X$. Then one can consider the quotient topological space $X/\sim$ and the quotient map $p:X\lo X/\sim$. In general, quotient spaces are not well behaved and it seems interesting to determine which topological properties of the space $X$ may be transferred to the quotient space $X/\sim$. For example, a quotient space of a simply connected or contractible space need not share those properties. Also, a quotient space of a locally compact space need not be locally compact.

Let $A_1,A_2,...,A_n$ be a finite collection of pairwise disjoint subspaces of a topological space $X$. Then by $X/(A_1,A_2,...,A_n)$ we mean the quotient space obtained from $X$ by identifying each of the sets $A_i$'s to a point. The main result of this paper is as follows:\\

{\bf Theorem A.} {\it If $X$ is a locally 1-connected metric space and $A_1,A_2,...,A_n$ are connected, locally path connected and compact pairwise disjoint subspaces of $X$, then the quotient space $X/(A_1,A_2,...,A_n)$ is locally 1-connected.}\\

The quasitopological fundamental group $\pi_1^{qtop}(X,x)$ of a based space $(X,x)$ is the familiar fundamental group $\pi_1(X,x)$ endowed with the
quotient topology with respect to the canonical function $\Omega(X,x)\rightarrow \pi_1(X,x)$ identifying
path components on the loop space $\Omega(X,x)$ with the compact-open topology (see \cite{Br}).
 It is known that this construction gives rise a homotopy invariant functor $\pi_1^{qtop}:hTop_*\lo qTopGrp$ from the homotopy category of based spaces to the category of quasitopological group and continuous homomorphism \cite{Br1}. By applying the above functor on the quotient map $p:X\lo X/\sim$, we have a continuous homomorphism $p_*:\pi_1^{qtop}(X,x)\lo\pi_1^{qtop}(X/\sim,*)$. Recently, the authors \cite{T} proved that if $X$ is a first countable, connected, locally path connected space and $A_1,A_2,...,A_n$ are disjoint path connected, closed subspaces of $X$, then the image of $p_*$ is dense in $\pi_1^{qtop}(X/(A_1,A_2,...,A_n))$. By using the results of \cite{T} and Theorem A we have the primary application of the main result of this paper as follows:\\

{\bf Theorem B.} {\it If $X$ is a locally 1-connected metric space and $A_1,A_2,...,A_n$ are disjoint connected, locally path connected and compact subspaces of $X$, then for each $a\in \bigcup_{i=1}^n A_i$ the continuous homomorphism $$p_*:\pi_1^{qtop}(X,a)\lo\pi_1^{qtop}(X/(A_1,A_2,...,A_n),*)$$
is an epimorphism.}\\

Also, by some examples, we show that $p_*$ is not necessarily onto if the hypotheses of the above theorem do not hold.

Finally, we give some applications of the above results to explore the properties of the fundamental group of the quotient space $X/(A_1,A_2,...,A_n)$. In particular, we prove that if $X$ is a separable, connected, locally 1-connected metric space and $A_1,...,A_n$ are connected, locally path connected and compact subsets of $X$, then $\pi_1(X/(A_1,A_2,...,A_n),*)$ is countable. Moreover, $\pi_1(X/(A_1,A_2,...,A_n),*)$ is finitely presented. Note that with the recent assumptions on $X$ and the $A_i$'s, $X/(A_1,A_2,...,A_n)$ is simply connected when $X$ is simply connected.
\newpage
\section{Notations and preliminaries}
For a pointed topological space $(X,x)$, by a path we mean a continuous map $\al : [0, 1]\lo X$. The
points $\al(0)$ and $\al(1)$ are called the initial point and the terminal point of $\al$, respectively.
A loop $\al$ is a path with $\al(0)=\al(1)$. For
a path $\al:[0,1]\lo X$, $\overline{\al}$ denotes a path such that $\ov{\al}(t)=\al(1-t)$, for all $t\in [0,1]$.
Denote $[0,1]$ by $I$, two paths $\al, \beta:I\lo X$ with the same initial and terminal points are called homotopic relative to end points if there exists a continuous map $F:I\times I\lo X$ such that
\begin{displaymath}
F(t,s)= \left\{
\begin{array}{lr}
\al(t)    &       s=0 \\
\beta(t)    &      s=1\\
\al(0)=\beta(0)   &  t=0\\
\al(1)=\beta(1)    &  t=1.
\end{array}
\right.
\end{displaymath}
The homotopy class containing a path $\al$ is denoted by $[\al ]$. Since most of the homotopies that
appear in this paper have this property and end points are the same, we drop the term ``relative homotopy'' for
simplicity. For paths $\al, \beta:I\lo X$ with $\al(1)=\beta(0)$, $\al*\beta$ denotes the concatenation of $\al$ and $\beta$, that is, a path
from $I$ to $X$ such that $(\al*\beta)(t)=\al(2t)$, for $0\leq t\leq 1/2$ and $(\al*\beta)(t)=\beta(2t-1)$,
for $1/2\leq t\leq1$.

For a space $(X,x)$, let $\Omega(X,x)$ be the space of based maps from $I$ to $X$ with the compact-open topology. A subbase
for this topology consists of neighborhoods of the form $\langle K,U\rangle=\{\gamma\in\Omega(X,x)\ |\ \gamma(K)\sub U\}$, where $K\sub I$ is compact
and $U$ is open in $X$. We will consistently denote the constant path at $x$ by $e_x$. The quasitopological fundamental group of a pointed space $(X,x)$ may be described as the usual
fundamental group $\pi_1(X,x)$ with the quotient topology with respect to the canonical map $\Omega(X,x)\lo\pi_1(X,x)$ identifying
homotopy classes of loops, denoted by $\pi_1^{qtop}(X,x)$. A basic account of quasitopological fundamental groups may be found in \cite{Br}, \cite{Br1} and \cite{C1}. For undefined notation, see \cite{M}.
\begin{definition} \cite{A}
A quasitopological group $G$ is a group with a topology such that inversion $G\lo G$, $g\lo
g^{-1}$, is continuous and multiplication $G \times G\lo G$ is continuous in each variable. A morphism of quasitopological groups is a continuous homomorphism.
\end{definition}
\begin{theorem} \cite{Br1}
$\pi_1^{qtop}$ is a functor from the homotopy category of based topological spaces to the category of quasitopological groups.
\end{theorem}
A space $X$ is {\it semi-locally simply connected} if for each point $x\in X$, there is an open
neighborhood $U$ of $x$ such that the inclusion $i : U \lo X$ induces the zero map $i_* : \pi_1(U,x)\lo \pi_1(X,x)$ or equivalently a loop in $U$ can be contracted inside $X$.
\begin{theorem} \cite{Br1}
Let $X$ be a connected and locally path connected topological space.
The quasitopological fundamental group $\pi_1^{qtop}(X, x)$ is discrete for some $x\in X$ if and
only if X is semi-locally simply connected.
\end{theorem}
\begin{theorem} \cite{T}
Let $A_1,A_2,...,A_n$ be disjoint path connected, closed subsets of a first countable, connected, locally path connected space $X$ such that $X/(A_1,A_2,...,A_n)$ is semi-locally simply connected, then for each $a\in\bigcup_{i=1}^nA_i$, $p_*:\pi_1(X,a)\lo\pi_1(X/(A_1,A_2,...,A_n,*)$ is an epimorphism.
 \end{theorem}
 In this note all homotopies are relative. Also, when $\al:[a,b]\lo X$ is a path, for brevity by $\hat{\al}$ we mean $\al\circ\vf$ where $\vf:I\lo [a,b]$ is a suitable linear homeomorphism.
 \newpage
\section{Constructing homotopies}
This section is dedicated to some technical lemmas about homotopy properties of loops.

 Assume that $\al$ is a path in $X$ and $t_0,t_1,...,t_n$ are real numbers such that $0=t_0<t_1<...<t_n=1$. Let $\al_i=\al|_{[t_i,t_{i+1}]}$, for $0\leq i\leq n-1$. Then we know that (\cite[Theorem 51.3]{M}) $$\al\simeq\hat{\al}_0*\hat{\al}_1*...*\hat{\al}_n.$$
\begin{corollary}
 Let $\al$ be a path in $X$ and $I_1,I_2,...I_n$ be closed disjoint subintervals of $I$. If $\al_j=\al|_{I_j}$ and $\bt_j:I_j\lo X$ are homotopic relative to end points for $j=1,...,n$, then $\al\simeq\bt$, where
 \begin{displaymath}
{\bt}(t)= \left\{
\begin{array}{lr}
\bt_j(t)   &       t\in I_j \\
\al(t)        &       otherwise. \\

\end{array}
\right.
\end{displaymath}
 \end{corollary}
  If $A$ is a subset of a topological space $X$, then we denote the complement of $A$ in $X$ by $A^c=X-A$.
\begin{lemma}
Let $\al$ be a loop in $X$ based at ${x_0}$. Let $\{t_i\}\sub (0,1)$ be an increasing sequence such that $t_i\rightarrow 1$ and $\{V_i\}$ form a nested neighborhood basis at $x_0$ such that $\al(t_i)\in V_i$ for every integer $i$. If $\bt_i:[t_i,1]\lo V_i$ is a path from $\al(t_i)$ to $x_0$ such that $\hat{\ov{\bt}}_i*\hat{\al}|_{[t_i,t_{i+1}]}*\hat{\bt}_{i+1}$ is a null loop in $V_i$, for every integer $i$, then $\al\simeq\al_i$, where
  \begin{displaymath}
{\al_i}(t)= \left\{
\begin{array}{lr}
\al(t)   &       t\leq t_i \\
\bt_i(t)        &      t\geq t_i. \\

\end{array}
\right.
\end{displaymath}
\end{lemma}
\begin{proof}
It is obvious that $\al_i\simeq\al_j$, for every integers $i,j$ and hence it suffices to prove $\al\simeq\al_1$ .
Let $\al_1\simeq\al_1$ by $F(t,s)=\al_1(t)$. Since $([t_1,1]\times\{0\})\cup(\{1\}\times I)$ is compact and $V_1$ is an open neighborhood of
 $F([t_1,1]\times\{0\}\cup\{1\}\times I)$, there exist $\e_1\in (0,1)$ such that $F(B_1)\sub V_1$, where $B_1=\{(x,y)\in I\times I\ |\ (1-t_1)y\leq\e_1 (x-t_1)\}$. Let $\eta_1:I\lo V_1$ be the path defined by $\eta_1(t)=F(t(t_1,0)+(1-t)(1,\e_1))$. By definition of $\eta_1$ and since $F(t,0)=\bt_1(t)$ for $t\geq t_1$, $\eta_1*\hat{\bt}_1\simeq e_{x_0}$ by a homotopy in $V_1$. Therefore
$$\eta_1*\hat{\al}|_{[t_1,t_2]}*\hat{\bt}_2\simeq\eta_1*\hat{\bt}_1*\hat{\ov{\bt}}_1*\hat{\al}|_{[t_1,t_2]}*\hat{\bt}_2\simeq
 e_{x_0}$$ by a homotopy in $V_1$
and hence there exists a continuous map $H_1:B_1\lo V_1$ such that \\
(i) $H_1(t(t_1,0)+(1-t)(1,\e_1))=\eta_1(t)$ for $t\in I$,\\
(ii) $H_1(t,0)=\al(t)$ for $t\in[t_1,t_2]$,\\
(iii) $H_1(t,0)=\bt_2(t)$ for $t\geq t_2$,\\
(iv) $H_1(1,t)=x_0$.\\
 Since $([t_2,1]\times\{0\})\cup(\{1\}\times I)$ is compact and $V_2$ is an open neighborhood of
 $H_1([t_2,1]\times\{0\}\cup\{1\}\times I)$, there exist $\e_2<\e_1\in (0,1)$ such that $H_1(B_2)\sub V_2$, where $B_2=\{(x,y)\in I\times I\ |\ (1-t_2)y\leq\e_2 (x-t_2)\}$. Let $\eta_2:I\lo V_2$ be the path defined by $\eta_2(t)=H_1(t(t_2,0)+(1-t)(1,\e_2))$. Then $\eta_1*\hat{\al}|_{[t_1,t_2]}*\eta_2^{-1}$ is null in $V_1$ by the homotopy $K_1=H_1|_{B_1-B_2}$(note that $B_1-B_2$ is homeomorphic to $I\times I$).
 Also, $\eta_2*\hat{\bt}_2$ is null in $V_2$ because $H_1(B_2)\sub V_2$ and hence
  $\eta_2*\hat{\al}|_{[t_2,t_3]}*\hat{\bt}_3\simeq e_{x_0}$ by homotopy $H_2:B_2\lo V_2$. By the similar way, for every $n\in\N$ we have the path $\eta_n:I\lo V_n$ by $\eta_n(t)=H_{n-1}(t(t_n,0)+(1-t)(1,\e_n))$ for which $\eta_{n-1}*\al|_{[t_{n-1},t_n]}*\eta_n^{-1}$ is null in $V_n$ by a homotopy $K_n=H_n|_{B_n-B_{n+1}}$, where $B_n=\{(x,y)\in I\times I\ |\ (1-t_n)y\leq\e_n (x-t_n)\}$ and $\varepsilon_n\rightarrow 0$ (see Figure 1).

 Define $G:I\times I\lo X$ by
\begin{displaymath}
{G}(t,s)= \left\{
\begin{array}{lr}
K_n(t,s)   &       (t,s)\in B_n-B_{n+1} \\
F(t,s)   &       o.w.\\

\end{array}
\right.
\end{displaymath}
Then $G(t,0)=\al(t)$ and $G(t,1)=\al_1(t)$. Let $(r_n,s_n)\rightarrow (1,0)$ and $U$ be an open neighborhood of $x_0$. There exist $n_0\in\N$ such that $V_{n_0}\sub U$. By the construction of $B_n$'s, there exist $n_1>n_0$ such that $(r_n,s_n)\in B_{n_1}$ for each $n>n_1$ which implies that
 $G(r_n,s_n)\in V_{n_1}\sub U$ for each $n>n_1$. Hence $G$ is continuous(preserving the other convergent sequences by $G$ is obvious).

\begin{figure}
\begin{center}
 \includegraphics[scale=1]{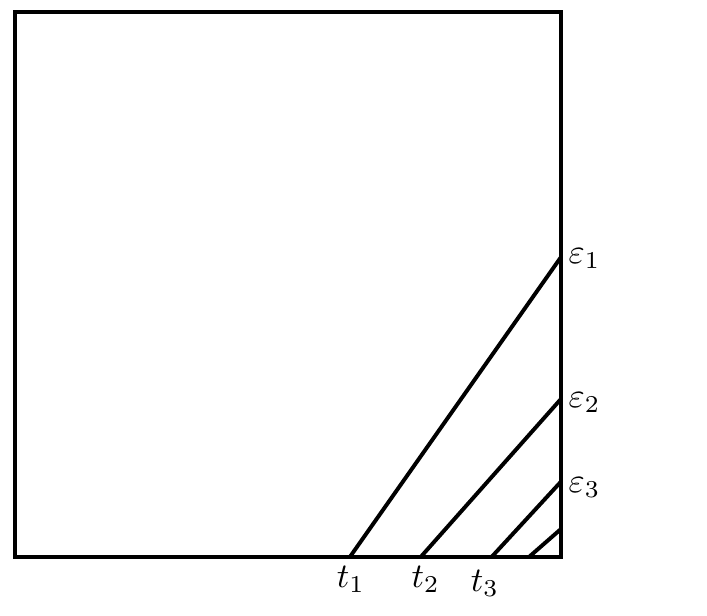}
  \caption{}\label{1}
\end{center}
\end{figure}

\end{proof}
In the following examples we show that the above lemma does not hold if$\ov{\bt}_i*\hat{\al}|_{[t_i,t_{i+1}]}*\bt_{i+1}$ is only null in $X$.
\begin{example} Let $X$ be the Griffiths' space: the one-point union $C(HE)\vee C(HE)$ of two copies of the cone over the Hawaiian earring as described in \cite{Gr}. Let $\theta_i$'s and $\lambda_i$'s be the consecutive loops of the right and left Hawaiian
earring, respectively. Then all loops $\theta_i$ and $\lambda_i$ are null by homotopies in $X$. Let $I_j=[1-\frac{1}{2^{j-1}},1-\frac{1}{2^j}]$, for every $j\in\N$ and define $\al:I\lo X$ by
 \begin{displaymath}
{\al}(t)= \left\{
\begin{array}{lr}
\theta_j\circ\varphi_{2j-1}(t)   &       t\in I_{2j-1}, \\
\lambda_j\circ\varphi_{2j}(t)   &        t\in I_{2j},\\

\end{array}
\right.
\end{displaymath}
where $\varphi_j$'s are suitable linear homeomorphisms from $I_j$ to $I$. Note that $\al$ is not null in $X$ \cite[Example 2.5.18]{S}. Let $t_i=1-\frac{1}{2^{i-1}}$, $\bt_{2i-1}=\theta_i\circ\varphi_{2i-1}$ and $\bt_{2i}=\lambda_i\circ\varphi_{2i}$. Then $\ov{\bt}_i*\hat{\al}|_{[t_i,t_{i+1}]}*\bt_{i+1}$ is homotopic to a null loop $\theta_j$ or $\lambda_j$ (depended on $i$ is even or odd). Also, every loop $\al_i$ defined as like as in Lemma 3.3 is null and hence $\al$ is not homotopic to $\al_i$, for every $i\in\N$.

\end{example}
\begin{lemma}
Let $\al$ be a loop in $X$ based at $x_0$, $\{x_0\}$ be closed and $\{V_i\}$ form a nested neighborhood basis at $x_0$. If $\al^{-1}(\{x_0\}^c)=\bigsqcup_{n=1}^{\infty} (a_n,b_n)$ and $\hat{\al}|_{[a_n,b_n]}\simeq e_{x_0}$ by a homotopy in $V_n$ ($\al([a_n,b_n])\sub V_n$), then $\al$ is nullhomotopic.
\end{lemma}
\begin{proof}
Let $I_n=[a_n,b_n]$ and $\al_n=\al|_{I_n}$. Assume $\hat{\al}_n\simeq e_{x_0}$ by a homotopy $F_n:I_n\times I\lo V_n$.
 Define $F:I\times I\lo V_1$ by\\
\begin{displaymath}
{F}(t,s)= \left\{
\begin{array}{lr}
F_n(t,s)   &       t\in I_n,\ \ for\ n\in\N \\
{x_0}   &  otherwise \\

\end{array}
\right.
\end{displaymath}
Since $F(t,0)=\al(t)$ and $F(t,1)=x_0$, it remains to show that $F$ is continuous. Let $(t_n,s_n)\rightarrow (t,s)$.
If $F(t,s)\neq x_0$, then $t_n\in(a_k,b_k)$, for a $k\in\N$ and every $n>k$. Continuity of $F_k$ implies that $F(t_n,s_n)=F_k(t_n,s_n)\rightarrow F_k(t,s)=F(t,s)$. If $F(t,s)=x_0$, then the construction of $F$, $F_n$'s and nested property of $V_n$'s imply
that $F(t_n,s_n)\in U$, for an open neighborhood $U$ of $x_0$ and sufficiently large $n$.
\end{proof}
By the next example we show that Lemma 3.4 does not hold if $\hat{\al}|_{[a_n,b_n]}$ is a null loop in $X$ instead of $V_n$.
\begin{example} Let $X$, $\al$, $\theta_i$ and $\lambda_i$ be defined as in Example 3.3 and $x_0$ be the common point of two cones. We have all the hypothesis of the lemma except $\theta_i$ and $\lambda_i$ are null in $X$ and $\al$ is not null.

\end{example}
\newpage
\section{The main result}
 This section dedicated to prove Theorem $A$ by introducing notions $\pi_1$-contained for subspaces of a given space $X$ and nested basis for a subspace $A\sub X$.
\begin{definition}
Let $X$ be a topological space, $Z,K$ be path connected subset of $Y\sub X$. We say that $Z$ is $\pi_1$-contained in $K$ at $Y$ if $i_*\pi_1(Z,x)\leq j_*\pi_1(K,x)$ for every $x\in Z\cap K$, where $i:Z\lo Y,\ j:K\lo Y$ are inclusion maps. Equivalently, every loop in $Z$ is homotopic to a loop in $K$ by a relative homotopy in $Y$.
\end{definition}
For a topological space $X$ and open subsets $G'\sub G\sub X$, we say $G'$ is simply connected in $G$ if every loop in $G'$ is nullhomotopic in $G$. If $A$ is a subset of a topological space $X$, then we denote the boundary of $A$ in $X$ by $\partial A$. Also, in a metric space $(X,d)$, by $S_r(a)$ we mean the open ball $\{x\in X|d(x,a)<r\}$.
\begin{lemma}
Let $X$ be a metric space and $A$ be a locally path connected, closed subset of $X$. Let $U\sub X$ be an open neighborhood of $a\in \partial A$  which is simply connected in an open subset $W$ of $X$. Then there exists an open neighborhood $V$ of the point $a$ such that $V\sub U$ and $A\cup V$ is $\pi_1$-contained in $A$ at $A\cup W$.
\end{lemma}
\begin{proof}
Since $X$ is metric, there exists $U'\sub U$ such that $\ov{U'}\sub U$. Since $A$ is locally path connected, there exists a neighborhood $V'\sub U'$ of $a$ such that $V'\cap A$ is path connected. There exists $r>0$ such that the closed ball $B_r(a)$ is contained in $V'$. Let $V$ be the open ball $S_r(a)$, then we show that $A\cup V$ is $\pi_1$-contained in $A$ at $A\cup W$. Let $\al$ be a loop in $A\cup V$ with base point in $A$. If $(r,s)$ is a component of $\al^{-1}(V\cap A^c)=\al^{-1}(A^c)$, then $\al(r),\al(s)\in \ov{V}\cap A\sub V'\cap A$. Since $V'\cap A$ is path connected, there exists a path in $V'\cap A$ from $\al(r)$ to $\al(s)$.\\
Case 1:

Let $\al^{-1}(A^c)=\bigcup_{n=1}^N(a_n,b_n)$. If for every $1\leq n\leq N$, $\bt_n:[a_n,b_n]\lo V'\cap A$ is a path from $\al(a_n)$ to $\al(b_n)$, then $\al|_{[a_n,b_n]}\simeq\bt_n$ by a homotopy in $W$ since $\hat{\al}|_{[a_n,b_n]}*\hat{\ov{\bt}}_n$ is a loop in $U$. Using Corollary 3.1, $\al$ is homotopic to a loop in $A$ by a homotopy in $A\cup W$.\\
Case 2:

Let $\al^{-1}(A^c)=\bigcup_{n=1}^{\infty}(a_n,b_n)$. Since $\al$ is uniformly continuous, there exists $\delta>0$ such that if $|t-s|<\delta$, then $d(\al(t),\al(s))<r/2$. Also, there exists $n_0\in\N$ such that for every $n>n_0$, $b_n-a_n<\delta$. Similar to Case 1, there are paths $\bt_i:[a_i,b_i]\lo V'\cap A$ such that $\bt_i\simeq \al|_{[a_i,b_i]}$, for $1\leq i\leq n_0$ by homotopies in $W$. Therefore $\al$ is homotopic to a loop $\al':I\lo A\cup V$ by a homotopy in $A\cup W$ and $\al'^{-1}(A^c)=\bigcup_{n=n_0+1}^{\infty}(a_n,b_n)$.

Let $F=\{a_n|n>n_0\}$ and $E=\{b_n| n>n_0\}$. If $t_1=inf\{{\al'^{-1}(\ov{A^c})\cap F}\}$ and $s_1=sup\{{[t_1,t_1+\delta]\cap E}\}$, then by definition of $E$ and $F$, $t_1,s_1$ exist and $s_1-t_1<\delta$ which implies that $\al'|_{[t_1,s_1]}\sub S_r(a)\sub V'$. Since $V'\cap A$ is path connected, there exists a path $\la_1:I\lo V'\cap A$ with $\la_1(0)=\al'(t_1)$ and $\la_1(1)=\al'(s_1)$
 such that $\hat{\al'}|_{[t_1,s_1]}\simeq\la_1$ by a homotopy in $W$ because $\hat{\al'}|_{[t_1,s_1]}*\ov{\la_1}$ is a loop in $V'\sub U$.
 Now, let $t_n=inf\{{[s_{n-1},1]\cap\al'^{-1}(\ov{A^c})\cap F}\}$ and $s_n=sup\{{[t_n,t_n+\delta]\cap E}\}$. Similarly, $\al'([t_n,s_n])\sub S_r(a)\sub V'$ and there exists a path $\la_n:I\lo V'\cap A$ with $\la_n(0)=\al'(t_n)$ and $\la_n(1)=\al'(s_n)$ such that $\hat{\al'}|_{[t_n,s_n]}\simeq\la_n$ by a homotopy in $W$ because $\hat{\al'}|_{[t_n,s_n]}*\ov{\la_n}$ is a loop in $V'\sub U$.
Since $I\setminus\cup_{i\leq n_0}(a_i,b_i)$ is compact and $\delta>0$, there exists $n_1>0$ such that $[s_{n_1-1},1]\cap\al^{-1}(\ov{A^c})\cap F=\emptyset$. Hence, we can replace $\hat{\al'}|_{[t_i,s_i]}$, by the path $\la_i$, for $1\leq i\leq n_1$, with homotopies in $W$. Therefore $\al$ is homotopic to a loop in $A$ by a homotopy in $A\cup W$, as desired.
\end{proof}
\begin{definition}
For a topological space $X$, a family $\mathcal{U}=\{U_n\}_{n\in \N}$ of open neighborhoods of a subset $A$ is called a {\bf  nested basis of $A$} if $U_{n+1}\sub U_n$ for all $n\in\N$ and for every open set $G$ containing $A$, there is $n_0\in \N$ such that $U_{n_0}\sub G$. Let a nested basis $\{V_n\}_{n\in\N}$ of $A$ be a refinement of a nested basis $\{U_n\}_{n\in\N}$ of $A$ i.e $V_n\sub U_n$ for each $n$. Then we say that $\{V_n\}_{n\in\N}$ is $\pi_1$-refinement in $A$ at $\{U_n\}_{n\in\N}$ if $V_n$ is $\pi_1$-contained in $A$ at $U_n$ for each $n$.
\end{definition}
There is two natural questions that whether a subset of a topological space $X$ has a nested basis and whether every nested basis has a $\pi_1$-refinement? Obviously, compact subsets of metric spaces have nested basis and for the second question we have the following theorem. We recall that a topological space $X$ is said to be locally 1-connected if it is locally path connected and locally simply connected.
\begin{theorem}
Let $X$ be a locally 1-connected metric space and $A$ be a locally path connected compact subset of $X$. Then  every nested basis $\{U_n\}_{n\in \N}$ of $A$ has a $\pi_1$-refinement
in $A$ at $\{U_n\}_{n\in \N}$.
\end{theorem}
\begin{proof}
Let $U_n=S_{\frac{1}{n}}(A)$ be the open ball with radius $1/n$ of $A$, then $\{U_n\}_{n\in \N}$ is a nested basis for $A$. Since $X$ is locally 1-connected, for every $n\in \N$ and $a\in \pa A$ there is an open neighborhood $G_n(a)\sub U_n$ at $a$ such that $G_n(a)$ is simply connected in $U_n$. By Lemma 4.2, there exists $G_n'(a)\sub G_n(a)$ such that $A\cup G_n'(a)$ is $\pi_1$-contained in $A$ at $U_n$. By local 0-connectivity of $X$ we can assume every $G_n'(a)$ is path connected. Since $A$ is compact, for every $n\in \N$ there exists $m_n\in \N$ and $a_n^1,a_n^2,...,a_n^{m_n}$ such that
$A\sub int(A)\bigcup(\cup_{i=1}^{m_n}{G'}_n^i(a_n^i) )$.

We claim for every $n\in \N$ there exists integer $k_n\geq n$ such that for every $m>k_n$ and every couple ${G'}_m^i(a_m^i),{G'}_m^j(a_m^j)$ with nonempty intersection, $A\cup {G'}_m^i(a_m^i)\cup {G'}_m^j(a_m^j)$ is $\pi_1$-contained in $A$ at $U_n$. By contradiction assume there exists no such $k_{n_0}$ for an $n_0\in \N$, then there are a sequence $\{s_n\}_{n\in \N}$, a couple ${G'}_{s_n}^i(a_{s_n}^i),{G'}_{s_n}^j(a_{s_n}^j)$ and loops $\al_n:I\lo A\cup {G'}_{s_n}^i(a_{s_n}^i)\cup {G'}_{s_n}^j(a_{s_n}^j)$ such that there is no loop in $A$ homotopic to $\al_n$ by a homotopy in $U_{n_0}$. Let $x_n\in Im\al_n\cap \pa A$. Since $\pa A$ is compact, there is a subsequence $\{x_{n_k}\}_{k\in \N}$ converging to $x\in \pa A$. Let $G$ be a simply connected neighborhood of $x$ in $U_{n_0}$. By Lemma 4.2 there exists an open neighborhood $x\in G'\sub G$ such that $A\cup G'$ is $\pi_1$-contained in $A$ at $U_{n_0}$. There exists $k_0\in \N$ such that for each $k>k_0$, ${G'}_{s_{n_k}}^i(a_{s_{n_k}}^i)\cup {G'}_{s_{n_k}}^j
(a_{s_{n_k}}^j)\sub G'\cup A$ which implies $Im\al_{n_k}\sub A\cup G'$, for $k>k_0$. Therefore $\al_{n_k}$ is homotopic to a loop in $A$ by a homotopy in $U_{n_0}$ which is a contradiction.

Let $V_n:=A\bigcup(\cup_{i=1}^{m_{k_n}}{G'}_{k_n}^i(a_{k_n}^i))$. We show that $V_n$ is $\pi_1$-contained in $A$ at $U_n$. Consider $\al:I\lo V_n$ as a loop at $a\in A$. Since $Im\al\sub int(A)\bigcup(\cup_{i=1}^{m_{k_n}}{G'}_{k_n}^i(a_{k_n}^i))$, there is the Lebesque number $\delta$ for this open cover. Let $n_0\in \N$ such that $\fr{1}{n_0}<\delta$, then there is $\{G_1, G_2,..., G_{n_0}\}\sub\{int(A), {G'}_{k_n}^1(a_{k_n}^1), {G'}_{k_n}^2(a_{k_n}^2),..., {G'}_{k_n}^{m_{k_n}}(a_{k_n}^{m_{k_n}})\}$ such that $\al([\fr{i-1}{n_0},\fr{i}{n_0}])\sub G_i$. Let $\al_i=\al|_{[\fr{i-1}{n_0},\fr{i}{n_0}]}\circ\vf_i$, where $\vf_i:[0,1]\lo [\fr{i-1}{n_0},\fr{i}{n_0}]$ is linear homeomorphism for $i=1,...,n_0$. Since $a=\al(0)\in G_1$ and $\al(\fr{1}{n_0})\in G_1\cap G_2$, if $\bt_1:I\lo A\cup G_2$ is a path from $\al(\fr{1}{n_0})$ to $a$ ($G_i$'s and $A$ are path connected), then $\al_1*\bt_1$ is a loop in $A\cup G_1\cup G_2$ which implies that $\al_1*\bt_1$ is homotopic to a loop $\gamma_1:I\lo A$ by a homotopy in $U_n$. Also, there exists $\bt_2:I\lo A\cup G_3$ such that $\ov{\bt_1}*\al_2*\bt_2$ is homotopic to a loop in $A$ by a homotopy in $U_n$ and similarly there are $\bt_i$'s and $\gamma_i$'s for $i=1, 2,..., n_0$ such that
$$\al\simeq(\al_1*\bt_1)*(\ov{\bt_1}*\al_2*\bt_2)*...*(\ov{\bt}_{n_0-2}*\al_{n_0-1}*\bt_{n_0-1})*(\ov{\bt}_{n_0-1}*\al_{n_0})\simeq\gamma_1*\gamma_2*...*\gamma_{n_0}$$
which implies that $\al$ is homotopic to a loop in $A$ by a homotopy in $U_n$, as desired. Note that if $\{V_n\}_{n\in \N}$ is not decreasing, then we can put ${V'}_n$ be the path component of $V_1\cap V_2\cap...\cap V_n$ that contains $A$.

Let $\{U_n'\}_{n\in\N}$ be an arbitrary nested basis of $A$. Since $\{U_n\}_{n\in\N}$ is a nested basis of $A$, there exists a subsequence $\{k_n\}_{n\in\N}$ of $\N$ such that $U_{k_n}\sub U_n'$ for every $n\in\N$. Let $\{V_{k_n}\}$ be the $\pi_1$-refinement of $\{U_{k_n}\}$.
Then $\{V'_n\}_{n\in\N}$ is a $\pi_1$-refinement of $\{U_n'\}_{n\in\N}$, where $V'_n=V_{k_n}$ for each $n$.
\end{proof}
\begin{remark} If $X$ is a metric space with a metric $d$, then for the quotient space $X/A$, the map $\ov{d}:X/A\times X/A\lo [0,+\infty)$ defined by
 \begin{displaymath}
{\ov d}(x,y)= \left\{
\begin{array}{lr}
min\{d(x,y), d(x,A)+d(y,A)\}    &        x,y\neq * \\
d(A,x)       &       x\neq y=* \\
0             &       x=y=*,\\
\end{array}
\right.
\end{displaymath}
where $*=p(A)$ and $p:X\rightarrow X/A$ is the natural quotient map, is a metric and the topology induced by this metric on $X/A$ is equivalent to the quotient topology on $X/A$ since $A$ is compact (see \cite[Section 3.1]{BU}).
\end{remark}
For a topological space $X$, a loop $\al$ in $X$ based at $x$ is called \emph{semi-simple} if $\al^{-1}(\{x\}^c)=(0,1)$ and is called \emph{geometrically simple} if $\al^{-1}(\{x\}^c)=(a,b)$ for $a,b\in [0,1]$. Also, every geometrically simple loop is homotopic to a semi-simple loop \cite{T}.
\begin{theorem}
Let $X$ be a locally 1-connected metric space and $A$ be a connected, locally path connected and compact subset of $X$. Then $X/A$ is locally 1-connected space.
\end{theorem}
\begin{proof}
 Since $X$ is locally path connected and $A$ is connected, $X/A$ is locally path connected. Clearly $X-A$ is an open subset of $X$ and hence it is locally 1-connected. Since $q:=p|_{X-A}:X-A\lo (X/A)-\{*\}$ is homeomorphism, $(X/A)-\{*\}$ is locally 1-connected.

Let $U$ be an open neighborhood of $*$. By the proof of Theorem 4.4, the nested basis $\{U_n=S_{\frac{1}{n}}(A)\}_{n\in \N}$ for $A$ has a refinement $\{V_n\}_{n\in \N}$ as its $\pi_1$-refinement in $A$ at $\{U_n\}_{n\in \N}$ such that $V_n$'s are path connected. We can assume that $U_1\sub p^{-1}(U)$ and show that $V:=p(V_1)$ is a desired open neighborhood. Let $\al:I\lo V$ be a loop at $*$, then we must show that $(i_V)_*([\al])=0$, where $i:V\hookrightarrow U$ is the inclusion map.\\
Case 1:

 Suppose that $\al$ is a semi-simple loop and $a_0=lim_{t\rightarrow 0}p^{-1}(\al(t))$ and $a_1=lim_{t\rightarrow 1}p^{-1}(\al(t))$ exist. Obviously $a_0,a_1\in\partial A$ and since $A$ is closed path connected, there exists a loop $\wt\al$ in $V_1$ with base point $a\in A$ such that $p\circ\wt\al\simeq\al$ by a homotopy in $V$.
Since $V_1$ is $\pi_1$-contained in $A$ at $U_1$, there is a loop $\beta$ in $A$ such that $\beta\simeq\wt\al$ by a homotopy in $U_1$. Therefore $\al\simeq p\circ\wt{\al}\simeq p\circ\bt\simeq e_*$ by homotopies in $p(U_1)\sub U$, as desired.\\
Case 2:

 Suppose that $\al$ is a semi-simple loop and at least one of the above limits does not exist.
 Let $\wt{\al}=q^{-1}(\al|_{(0,1)})$. Without loss of generality, assume $a_0=lim_{t\rg 0}\ov\al(t)$ exists and $a_1\in A$ is a limit point of $\ov\al((1/2,1))$. The point $a_1$ is a limit point and $\al$ is continuous, hence there is an increasing sequence $\{t_n\}_{n\in \N}$ in $(1/2,1)$ such that $\wt{\al}([t_n,1))\sub V_n$, for every $n\in \N$ and $t_n\rightarrow 1$. Since the $V_n$'s are path connected, there are paths $\gamma_n:[t_n,1]\lo V_n$ from $\wt{\al}(t_n)$ to $a_1$.
  If $\bt_n:=p\circ\gamma_n$, then $Im\bt_n\sub p(V_n)$ and $\ov{\bt}_n*\hat{\al}|_{[t_n,t_{n+1}]}*\bt_{n+1}$ is null in $p(U_n)$ since $\ov{\gamma}_n*\hat{\wt{\al}}|_{[t_n,t_{n+1}]}*\gamma_{n+1}$ is a loop in $V_n$. Let
  \begin{displaymath}
{\al_n}(t)= \left\{
\begin{array}{lr}
\al(t)   &       t\leq t_n \\
\bt_n(t)        &      t\geq t_n. \\

\end{array}
\right.
\end{displaymath}
By Lemma 3.2, $\al\simeq\al_n$ for each $n$ and by Case 1, $\al_n$'s are nullhomotopic in $p(U_1)$ which implies that $\al$ is nullhomotopic.

Case 3:

If $\al$ is not a semi-simple loop in $V$, then $\al^{-1}(X/A-\{*\})=\bigcup_{i\in \N}(a_i,b_i)$. Obviously $\al|_{[a_i,b_i]}$ is null in a $p(U_{n_i})$, where $n_i=max\{n\in\N|\al([a_i,b_i])\sub p(V_n)\}$. By Lemma 3.4, $\al$ is null.

Similarly, the result holds if $\al$ is a geometrically simple loop. Also, in Case 3, if $\al^{-1}(X/A-\{*\})=\bigcup_{i=1}^n(a_i,b_i)$, by Corollary 3.1 the result holds.
\end{proof}\ \\
{\bf Proof of Theorem A}. By Remark 4.5 and Theorem 4.6, $X/A_1$ is a locally 1-connected metric space, hence we can use Theorem 4.6 again and by induction the result holds.\\

In the following examples, we show that locally path connectedness and locally simply connectedness of $X$ are necessary conditions in Theorem A.
\begin{example}
Let $Y=\{(x,y)\in\mathbb{R}^2\ |\ x^2+y^2=(1/2+1/n)^2,\ n\in\N\}$, $A=\{(x,y)\in\mathbb{R}^2\ |\ x^2+y^2=1/4\}\cup\{(x,0)\in\mathbb{R}^2\ |\ 1/2\leq x\leq 1\}$ and $X=Y\cup A$ with $a=(1,0)$ as the base point (see Figure 2). Then $X$ is connected, locally simply connected metric space and $A$ is compact, connected and locally path connected subset of $X$, but obviously, $X$ is not locally path connected. Since for every open neighborhood $U\sub X/A$ of $*$ there exists $N_0\in\N$ such that for $n>N_0$, $p^{-1}(U)$ contains circles with radius $1/2+1/n$, $X/A$ is the
union of a null sequence of simple closed curves meeting in a common point that is homeomorphic to the Hawaiian
earing space. Hence $X/A$ is not locally simply connected.
\begin{figure}
\begin{center}
 \includegraphics[scale=1]{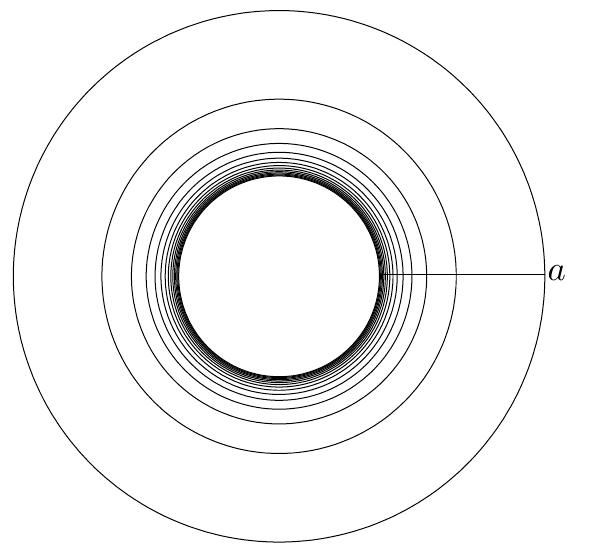}
  \caption{}\label{1}
 \end{center}
\end{figure}
\end{example}

\begin{example}
Let $X=\{(x,y)\in\mathbb{R}^2\ |\ 1/4\leq x^2+y^2\leq 1\}\setminus\{(-1/2-1/n,0)\in\mathbb{R}^2\ |\ n\in\N\}$ and $A=\{(x,y)\in\mathbb{R}^2\ |\ x^2+y^2=1/4\}\cup\{(x,0)\in\mathbb{R}^2\ |\ 1/2\leq x\leq 1\}$.
Then $X$ is a connected, locally path connected metric space that is not locally simply connected and $A$ is a compact, connected and locally path connected subset of $X$.
Similar to Example 4.7, it can be proved that $X/A$ is homeomorphic to $Y=D^1\setminus\{(-1/n,0)\ |\ n\in\N\}$ which is not locally simply connected.
\end{example}
\newpage
\section{Some Applications to fundamental groups}
In this section, first we give a short proof of Theorem B and then prove some properties of the fundamental group of the quotient space $X/(A_1,A_2,...,A_n)$.\\
\ \ \\
{\bf Proof of Theorem B.} By Theorem 2.4, it suffices to prove that $X/(A_1,A_2,...,A_n)$ is semi-locally simply connected. But by Theorem A, we have $X/(A_1,A_2,...,A_n)$ is locally 1-connected which is more.
\begin{remark}
In Examples 4.7 and 4.8, $p_*$ is not onto since $\pi_1(X,a)$ is a free group on countably many generators and $\pi_1(X/A,*)$ is uncountable. This shows that the hypothesis locally 1-connectedness is necessary for Theorem B.
\end{remark}
We know that any homomorphic image of a finitely generated group is also finitely generated. Hence we have the following consequence of Theorem B.
\begin{corollary}
Let $X$ be connected, locally 1-connected metric space and $A_1,A_2,...,A_n$ be compact, connected, locally path connected subsets of $X$. If $\pi_1(X,a)$ is finitely generated, then so is $\pi_1(X/(A_1,A_2,...,A_n),*)$.
\end{corollary}
By a theorem of Shelah \cite{sh}, if $X$ is a compact, connected, locally path connected metric space and $\pi_1(X,x)$ is countable, then $\pi_1(X,x)$ is finitely generated. Using this and Theorem B, we have
\begin{corollary}
Let $X$ be compact, connected, locally 1-connected, metric space and $A_1,A_2,...,A_n$ be closed, connected, locally path connected subsets of $X$.
If $\pi_1(X/(A_1,A_2,...,A_n),*)$ is not finitely generated, then $\pi_1(X,a)$ is uncountable.
\end{corollary}
J. Dydak and Z. Virk \cite{D} proved that if $X$ is a connected, locally path connected metric space and $\pi_1(X,x)$ is countable, then $\pi_1(X,x)$ is a finitely presented group. Note that a homomorphic image of a finitely presented group is not necessarily a finitely presented group.
\begin{corollary}
Let $X$ be connected, locally 1-connected metric space and $A_1,A_2,...,A_n$ be compact, connected, locally path connected subsets of $X$. If $\pi_1(X,a)$ is countable, then $\pi_1(X/(A_1,A_2,...,A_n),*)$ is finitely presented.
\end{corollary}
J.W. Cannon and G.R. Conner \cite{CC} proved that if $X$ is a one-dimensional metric space which is connected, locally path connected and semi-locally
simply connected (or equivalently if X admits a universal covering space), then $\pi_1(X,x)$ is a free group. Therefore by Theorem A we have
\begin{corollary}
If $X$ is a one-dimensional metric space which is connected, locally 1-connected, and $A_1,A_2,...,A_n$ are compact, connected, locally path connected subsets of $X$, then $\pi_1(X/(A_1,A_2,...,A_n),*)$ is a free group.
\end{corollary}

\begin{theorem}
If $X$ is a separable, connected, locally 1-connected metric space and $A_1,A_2,...,A_n$ are compact, connected, locally path connected subsets of $X$, then $\pi_1(X/(A_1,A_2,...,A_n),*)$ is countable. Moreover, $\pi_1(X/(A_1,A_2,...,A_n),*)$ is finitely presented.
\end{theorem}
\begin{proof}
Using Theorem A and \cite[Lemma 5.6]{CC}, we have $\pi_1(X/(A_1,A_2,...,A_n),*)$ is countable and by \cite{D}, $\pi_1(X/(A_1,A_2,...,A_n),*)$ is finitely presented.
\end{proof}
J. Pawlikowski \cite{PW} proved that compact, connected, locally path connected metric spaces with countable fundamental groups are semi-locally simply connected. Therefore we have
\begin{corollary}
If $X$ is a compact, connected, locally path connected metric space
 and $A_1,A_2,...,A_n$ are connected, locally path connected closed subsets of $X$, then $p_*:\pi_1(X,a)\lo\pi_1(X/(A_1,...,A_n),*)$ is an epimorphism if one of the following conditions holds\\
(i) $\pi_1(X/(A_1,A_2,...,A_n),*)$ is countable,\\
(ii) $\pi_1(X/(A_1,A_2,...,A_n),*)$ is finitely generated,\\
(iii) $\pi_1(X/(A_1,A_2,...,A_n),*)$ is finitely presented,\\
(iv) $X/(A_1,...,A_n)$ has universal covering.
\end{corollary}
\begin{proof}
By Assumptions, Remark 4.5 and continuity of $p:X\lo X/(A_1,A_2,...,A_n)$), $X/(A_1,A_2,...,A_n)$ is a compact, connected, locally path connected metric space and the result follows from Theorem 2.4.
\end{proof}
\begin{remark}
J. Brazas \cite{Br1} introduced a new topology on $\pi_1(X,x)$ coarser than the topology of  $\pi_1^{qtop}(X,x)$ and proved that fundamental groups by this new topology are topological groups, denoted by $\pi_1^\tau(X,x)$. Since the topology of $\pi_1^\tau(X,x)$ is coarser than $\pi_1^{qtop}(X,x)$, the results of \cite{T} remain true and we have a similar result to Theorem B for $\pi_1^\tau(X,x)$ and $\pi_1^\tau(X/(A_1,...,A_n),*)$.
\end{remark}

\subsection*{Acknowledgements}
The authors would like to thank the referee for the valuable comments and suggestions which improved the manuscript and made it more readable.


\end{document}